\documentclass{article}
\usepackage{geometry}
\usepackage{graphicx,tikz,caption,subcaption}
\usepackage{fullpage}
\usepackage{hyperref}
\usepackage{enumitem,verbatim}
\usepackage{hyperref}
\usepackage{amsmath}
\usepackage{amssymb}
\usepackage{amsfonts}
\usepackage{amsthm}
\usepackage{authblk}
\usepackage[capitalise]{cleveref}
\usepackage{geometry}
\usepackage{float}
\usepackage{setspace}
\newtheorem{theorem}{Theorem}[section]
\newtheorem{lemma}[theorem]{Lemma}

\newtheorem{claim}[theorem]{Claim}

\newtheorem{remark}{Remark}
\newtheorem*{construction}{Construction}

\newtheorem{observation*}{Observation}[section]
\newtheorem{proposition}[theorem]{Proposition}

\newtheorem{question}[theorem]{Question}
\newcommand{\ex}{\mathrm{ex}}

\newenvironment{pf}{%
    \begin{proof}[Proof]
}{%
    
    \end{proof}
}

\newenvironment{mycase}[1]{%
    \par\addvspace{\medskipamount}%
    \noindent\textbf{#1}\quad%
}{%
    \par\addvspace{\medskipamount}%
}

\title{On generalized Tur\'an problems with bounded matching number}
 \author[1]{Yisai Xue}
 \affil[1]{School of Mathematics and Statistics, Ningbo University, Ningbo, China}

 \author[2,3]{Liying Kang\thanks{Research was partially supported by the National Nature Science Foundation of China (grant number 12331012)}\thanks{\textit{Corresponding author. Email address: lykang@shu.edu.cn (L. Kang)}}}
 \affil[2]{Department of Mathematics, Shanghai University, Shanghai 200444, P.R. China}
 \affil[3]{Newtouch Center for Mathematics of Shanghai University, Shanghai, China, 200444}

\date{}

\begin{document}

\maketitle

\begin{abstract}
  The generalized Tur\'an number $\mathrm{ex}(n, H, \mathcal{F})$ is defined as the maximum number of copies of a graph $H$ in an $n$-vertex graph that does not contain any graph $F \in \mathcal{F}$.
   Alon and Frankl initiated the study of Tur\'an problems with a bounded matching number.
   In this paper, we establish stability results for generalized Tur\'an problems with bounded matching number.
   Using the stability results, we provide exact values of $\ex(n,K_r,\{F,M_{s+1}\})$ for $F$ being any non-bipartite graph or  a path on $k$ vertices.

\bigskip

\medskip

\noindent{\bf AMS (2000) subject classification:}  05C35
\end{abstract}

\section{Introduction}

  The Tur\'an problem is a classic problem in extremal graph theory that concerns the maximum number of edges in a graph without any copies of a specified family.
  In 1941, Tur\'an proved that if a graph does not contain a complete subgraph $K_{k+1}$, then the maximum number of edges it can contain is given by the \textit{Tur\'an graph} $T_k(n)$, i.e., the complete balanced $k$-partite graph on $n$ vertices.
   Erd\H{o}s and Gallai \cite{gallai1959maximal} established another fundamental result by determining $\ex(n,M_{s+1})$, where $M_{s+1}$ is a matching consisting of $s+1$ independent edges.

	For two graphs $G$ and $H$, the {\em disjoint union} of $G$ and $H$ is denoted by $G\cup H$. The \textit{join} of $G$ and $H$, denoted by $G\vee H$, is the graph obtained from $G \cup H$ by adding all possible edges between $G$ and $H$.
   The study initiated by Alon and Frankl \cite{alon2024turan} focused on Tur\'an problems with a bounded matching number, laying the groundwork for our investigation.
   They proved that for $n \geq 2 s+1$ and $k \geq 2$,
   $$\ex(n,\{K_{k+1}, M_{s+1}\})=\max \{e(T_k(2s+1)),e(T_{k-1}(s)\vee I_{n-s})\},$$
   where $I_{n-s}$ is an independent set of size $n-s$.

   Later, Gerbner \cite{gerbner2024turan} investigated Tur\'an problems with arbitrarily forbidden subgraphs of chromatic number at least 3 and bounded matching number.

  Given a graph $H$ and a family of graphs $\mathcal{F}$, the maximum possible number of copies of $H$ in an $n$-vertex graph that does not contain any copy of $F \in \mathcal{F}$ is denoted by $\mathrm{ex}(n, H, \mathcal{F})$ and is called \textit{generalized Tur\'an number}.
  We denote by $\mathcal{N}(H, G)$ the number of copies of $H$ in $G$.
  Moreover, let $\mathrm{EX}(n, H,\mathcal{F})$ be the set of $n$-vertex $\mathcal{F}$-free graphs that satisfying $\mathcal{N}(H,G)=\ex(n, H,\mathcal{F})$.
  If $\mathcal{F}=\{F\}$, we simply denote it by $\ex(n, H, F)$.
  Note that when $H=K_2$, $\ex(n, K_2, F)$ is the classical Tur\'an number.
  The concept of the generalized Tur\'an number was formally introduced by Alon and Shikhelman \cite{alon2016many} in 2016, and further related results can be found in the literature \cite{gao2023counting,gerbner2020generalized,gishboliner2020generalized,luo2018maximum,xue2024stability}.

  Our first result concerns the case where the chromatic number of $F$ is at least 3.
  Given a graph $F$ with $\chi(F)=r \ge 3$, let $\mathcal{D}(F)$ be the set of all graphs that can be obtained from $F$ by deleting an independent set of $F$.

\begin{theorem}\label{thm:stability}
  Let $r\geq 3$ and $z>0$ be constants.
  Suppose $s$ is an integer and $F$ is a graph with $\chi(F)\ge r$.
  Assume there exists $n_0=n_0(s,r,z)$ such that for every graph $G$ on $n>n_0$ vertices satisfies that
\begin{align*}
  G \text{~is~} \{F,M_{s+1}\}\text{-free}, \text{~and\quad}  \mathcal{N}(K_r,G)\geq (n-s)\ex(s, K_{r-1},\mathcal{D}(F))- z n,
\end{align*}
  then the following holds.
\begin{itemize}
    \item[\rm (i)] If $\ex(s, K_{r-1},\mathcal{D}(F))-\ex(s-1, K_{r-1},\mathcal{D}(F))>z$, then $G$ has an independent set of size $n-s$.
    \item[\rm (ii)] Moreover, if $\ex(s, K_{r-1},\mathcal{D}(F))-\mathcal{N}(K_{r-1},H)> z$ holds for all $\mathcal{D}(F)$-free graphs $H$ with $s$ vertices that are not in $\mathrm{EX}(s, K_{r-1},\mathcal{D}(F))$, then $G$ is a subgraph of $D\vee I_{n-s}$ for some $D\in\mathrm{EX}(s, K_{r-1},\mathcal{D}(F))$.
\end{itemize}
\end{theorem}

 Ma and Hou \cite{ma2023generalized} determined the exact value of $\ex(n, K_r,\{K_{k+1}, M_{s+1}\})$ and gave an asymptotic value of $\ex(n, K_r,\{F, M_{s+1}\})$ for general $F$ with an error term $O(1)$. Moreover, they asked whether the following holds.

\begin{question}\label{question}
    For every graph $F$ with $\chi(F) \geq 3$ and $r \geq 3$, if $\ex(s, K_{r-1}, \mathcal{D}(F))>0$ and $D\in \mathrm{EX}(s, K_{r-1}, \mathcal{D}(F))$, then, for sufficiently large $n$,
$$\ex(n, K_r,\{F, M_{s+1}\})=\mathcal{N}(K_r,D\vee I_{n-s}).$$
\end{question}

  The following example gives a negative answer to \cref{question}.

  For integers $p, k$ and $r=3$, let $s=k(2p-1)+1$ and
   $F=P_{2p}\vee I_p$.
   Then $\mathcal{D}(F)=\{P_{2p}\vee I_t:0\le t\le p-1\}\cup\{L\vee I_p:L \text{ is a linear forest with at least $p$ vertices and at most $2p-1$ edges}\}$.
   Utilizing results  by Faudree and Schelp \cite{faudree1975path} and independently Kopylov \cite{kopylov1977maximal}, which determine $\ex(n,K_2,P_\ell)$ exactly and describe the extremal graphs, we have  $$\ex(s,K_2,\mathcal{D}(F))=\ex(s-1,K_2,\mathcal{D}(F))=\ex(s-1,K_2,P_{2p})=k\binom{2p-1}{2}.$$
   Let $D_F(n,r)$ be the graph in $\mathrm{EX}(n, K_{r-1},\mathcal{D}(F))$ with the maximum number of $r$-cliques.
   Then $D_F(s,r)=kK_{2p-1}\cup K_1$ for $F=P_{2p}\vee I_p$ and $r=3$.
   A simple calculation shows that
\begin{align*}
    \ex(n,K_r,\{F,M_{s+1}\})
    &\ge\mathcal{N}(K_r,(kK_{2p-1})\vee I_{n-s+1})\\[2mm]
    &>\mathcal{N}(K_r,(kK_{2p-1}\cup K_1)\vee I_{n-s})\\[2mm]
    &=\mathcal{N}(K_r,D_F(s,r)\vee I_{n-s}).
\end{align*}

  The following theorem suggests to use condition ``$\mathrm{ex}(s, K_{r-1},\mathcal{D}(F))>\mathrm{ex}(s-1, K_{r-1},\mathcal{D}(F))$''  instead of the condition  ``$\mathrm{ex}(s, K_{r-1},\mathcal{D}(F))>0$''.

\begin{theorem}\label{cor1}
  Let $r \geq 3$ and  $n$ be sufficiently large.
  For every graph $F$ with  $\mathrm{ex}(s, K_{r-1},\mathcal{D}(F))>\mathrm{ex}(s-1, K_{r-1},\mathcal{D}(F))$ and $\chi(F)\geq r$, we have
$$
\ex(n, K_r,\{F, M_{s+1}\})=\mathcal{N}(K_r,D_F(s,r)\vee I_{n-s}).
$$
\end{theorem}

\begin{remark}
   The above example also shows that the condition ``$\ex(s, K_{r-1},\mathcal{D}(F))>\ex(s-1, K_{r-1},\mathcal{D}(F))$'' is necessary.
\end{remark}

For graph $F$ with $\chi(F)>r \geq 3$, we have the following result.

\begin{theorem}\label{cor2}
 If $\chi(F)> r\geq 3$, $s$ and $n$ are sufficiently large, then
$$\mathrm{ex}(n, K_r,\{F, M_{s+1}\})=\mathcal{N}(K_r,D_F(s,r)\vee I_{n-s}).$$
\end{theorem}

   Next, we consider the case that $\chi(F)=2$.
   Let $F$ be a bipartite graph and  $p = p(F)$ be the smallest possible order of a colour class in a proper two-colouring of $F$.
   When $p>s$,  Ma and Hou \cite{ma2023generalized} showed that $\mathrm{ex}(n,K_r,\{F,M_{s+1}\})=\mathrm{ex}(n,K_r,M_{s+1})$.
   For $p\leq s$, they gave the value of $\ex(n,K_r,\{F,M_{s+1}\})$ with an error term $O(1)$.
   Using the stability lemma, we provide an explicit upper bound.

\begin{theorem}\label{thm:thm2.3}
  Let $s$ be an integer.
  Suppose $F$ is a bipartite graph with $p = p(F)\le s$.
 Let $t=\binom{2s}{p}(q-1)+2s+1-p$.
 We have
$$\binom{p-1}{r}+\binom{p-1}{r-1}(n-p+1)\le\ex(n, K_r,\{F, M_{s+1}\})\le \binom{t+p-1}{r}+\binom{p-1}{r-1}(n-t-p+1).$$
\end{theorem}

   Furthermore, we provide exact values for the case when $F$ is a path and $n$ is sufficiently large.
   The constructions of $G_1$, $\ldots$, $G_6$ will be described in \cref{section:bipartite}.


\begin{theorem}\label{2l}
Let $s-p+1=a(p-1)+b$, where $0\leq b\leq p-2$.
For $n$ is sufficiently large, we have
\begin{align*}
  \ex(n,K_r,\{P_{2p},M_{s+1}\})=  \max\{\mathcal{N}(K_r,G_1),\mathcal{N}(K_r,G_2)\}.
\end{align*}
\end{theorem}

 \begin{theorem}\label{2l+1}
 Let $p\le s$ and $n$ be a sufficiently large integer.
 Suppose that $s-p+1=cp+d$, where $0\leq d\leq p-1$.
\begin{itemize}
    \item If $d=0$, then
 $$\ex(n,K_r,\{P_{2p+1},M_{s+1}\})=\max\{\mathcal{N}(K_r,G_3),\mathcal{N}(K_r,G_4)\}.$$
 \item If $1\leq d\leq p-2$, then
\begin{align*}
  \ex(n,K_r,\{P_{2p+1},M_{s+1}\})=  \max\{\mathcal{N}(K_r,G_3),\mathcal{N}(K_r,G_4),\mathcal{N}(K_r,G_5),\mathcal{N}(K_r,G_6)\}.
\end{align*}
\item If $d=p-1$, then
\begin{align*}
  \ex(n,K_r,\{P_{2p+1},M_{s+1}\})=  \max\{\mathcal{N}(K_r,G_3),\mathcal{N}(K_r,G_4),\mathcal{N}(K_r,G_6)\}.
\end{align*}
\end{itemize}
\end{theorem}

\begin{remark}
 Unlike the case of $\ex(n,K_2,\{P_{k},M_{s+1}\})$ studied in \cite{gerbner2024turan}, the extremal graphs involved in \cref{2l} and \cref{2l+1} are both possible, depending on the values of $r$, $p$ and $s$.
\end{remark}

\section{Preliminary}
 All the graphs involved in this paper are simple graphs. Let $G$ be a graph. We will denote the set of vertices of $G$ by $V(G)$ and the set of edges by $E(G)$, and define $|G| := |V(G)|$ and $e(G):=|E(G)|$. For any $v\in V(G)$, let $N_G(v)$ be the set of the neighbours of $v$ in $G$, $d_G(v) :=|N_G(v)|$, and $\delta(G) :=\min\limits_{v\in V(G)} d_G(v)$. We may omit the graph $G$ when it is clear.

  For a subset $U\subseteq V(G)$ we define the common neighbourhood of $U$ as $N(U)=\bigcap_{u \in U} N(u)$.
  For any $U \subseteq V(G)$, let $G[U]$ be the subgraph induced by $U$ whose edges are precisely the edges of $G$ with both ends in $U$.
  Let $S\subseteq V(G)$ and define $G-S$ be $G[V(G)\backslash S]$.

  We use $K_n$, $C_n$, $P_n$ to denote the complete graph (clique), cycle, and path on $n$ vertices respectively.
  The \textit{matching number} $\nu(G)$ is the number of edges in a maximum matching of $G$.
    If $M$ is a matching, the two ends of each edge of $M$ are said to be \textit{matched} under $M$.

  The famous Tutte-Berge theorem characterizes the structure of graphs with bounded matching number.
\begin{theorem}[Tutte-Berge Theorem \cite{lovasz2009matching}]\label{Tutte}
  A graph $G$ is $M_{s+1}$-free if and only if there is a set $B \subseteq V(G)$ such that all the components $G_1, \ldots, G_m$ of $G-B$ are odd (i.e. $|V(G_i)| \equiv 1(\bmod~2)$ for $i \in[m]$), and
$$|B|+\sum_{i=1}^m \frac{|V(G_i)|-1}{2}=s.$$
\end{theorem}

 To prove our result, we will use the following results.

\begin{proposition}[\cite{alon2016many}]\label{prop:erdos-stone}
  For any graph $H$, $\ex(n, K_r, H)=\Omega(n^r)$ if and only if $\chi(H)>r$.
\end{proposition}

\begin{lemma}[Dirac \cite{dirac1952some}]\label{Dirac}
  Let $G$ be a connected graph. If $P$ is a longest path of $G$ with ends $u$ and $v$, then
$$
|P| \geq \min \{|G|, d(u)+d(v)+1\} .
$$
\end{lemma}

\begin{lemma}[Chakraborti and Chen \cite{chakraborti2021many}]\label{2021C}
  Let $r$, $w$, $x$, $y$, and $z$ be non-negative integers such that $r \geq 2$, $x+y=w+z$, $x \geq w$, $x \geq z$, and $x \geq r$.
  Then, $$\binom{x}{r}+\binom{y}{r} \geq\binom{w}{r}+\binom{z}{r}.$$
  Moreover, the inequality is strict if $x>w$ and $x>z$.
\end{lemma}

\begin{lemma}[Vandermonde's identity]\label{lmm:Vandermonde}
  For any nonnegative integers $r$, $m$ and $n$, we have
  $$\binom{m+n}{r}=\sum_{i=0}^r\binom{m}{i}\binom{n}{r-i}.$$
\end{lemma}

\section{Forbidding non-bipartite graphs}\label{section3}

  In this section, we first prove \cref{thm:stability}, and then use it to prove \cref{cor1} and  \cref{cor2}.

\begin{proof}[\textbf{Proof of \cref{thm:stability}.}]
  Let $F$ be a graph such that $\chi(F) \geq r\geq 3$ and $\mathrm{ex}(s, K_{r-1},\mathcal{D}(F))-\mathrm{ex}(s-1, K_{r-1},\mathcal{D}(F))>z$.
  Consider a graph $G$  on $n$ vertices that is $\{F, M_{s+1}\}$-free and satisfies $$\mathcal{N}(K_r,G)\geq (n-s)\mathrm{ex}(s, K_{r-1},\mathcal{D}(F))-z n.$$
  Additionally, assume that $G$ does not contain isolated vertices.

  By Tutte-Berge Theorem (\cref{Tutte}), there is a set $B \subseteq V(G)$ such that all the components $G_1, \ldots, G_m$ of $G-B$ are odd, and
\begin{align}\label{eq:B}
 |B|+\sum_{i=1}^m \frac{|V(G_i)|-1}{2}=s.
\end{align}

  We define $W$ as the set of isolated vertices in $G-B$.
  Let $t=\sum\limits_{|V(G_i)|>1}|V(G_i)|$.
  It follows from \eqref{eq:B} that $t\le 3s$. Thus, we have
\begin{align}\label{eq:W}
  | W | = n - t - | B | \geq n - 4s.
\end{align}

  Since each $w \in W$ only has neighbours in $B$, we have $d(w) \leq | B | \leq s$.
  Then $W$ can be divided in to disjoint union $W_1\cup\cdots\cup W_s$ according to the degrees of its vertices, where
  $$W_i:=\{w\in W:d(w)=i\}.$$

  For any $i\in[s]$, we define $$\mathcal{R}(i):=\{R\subseteq B: R=N(v)\text{ for some } v\in W_i\}$$ as the set of $i$-subsets of $B$ which is the neighbourhood of some vertex $v$ in $W_i$.
  We divide $\mathcal{R}(i)$ into two parts based on the number of common neighbours of $R$ in $W_i$.
  Let
\begin{align*}
  \mathcal{R}'(i):=\{R\in \mathcal{R}(i):|N(R)\cap W_i|\geq |F|\} \text{\quad and\quad } \mathcal{R}''(i):=\{R\in \mathcal{R}(i):|N(R)\cap W_i|< |F|\}.
\end{align*}

     Let $W_i'=\{v\in W_i: N(v)\in \mathcal{R}'(i)\}$ and $W_i''=W_i\setminus W_i'$.
     The following claim gives the properties of $W_i'$ and $W_i''$.
\begin{claim}\label{cl:Wi'}
 (a) For any $v\in W_i'$, $G[N(v)]$ is $\mathcal{D}(F)$-free;
   (b) $|W_i''|<\binom{s}{i}|F|$.
\end{claim}

\begin{pf}
  (a).
   Let $v\in W_i'$ and $R_v$ be the neighbourhood set of $v$.
  Then $|N(R_v)|\geq |N(R_v)\cap W_i|\geq |F|$ by the definitions of $W_i'$ and $\mathcal{R}'(i)$.
  Therefore, $G[R_v]$ is $\mathcal{D}(F)$-free. Otherwise, $G$ contains a copy of $F$ since $R_v$ has at least $|F|$ common neighbours.

   (b). By the definition of $\mathcal{R}''(i)$, for each $R\in \mathcal{R}''(i)$, there are at most $|F|-1$ vertices in $W_i$ whose neighbourhood set is  $R$.
   Since there are at most $\binom{s}{i}$ choices of $R$,  $|W_i''|\le \binom{s}{i}(|F|-1)<\binom{s}{i}|F|$.
\end{pf}

  Note that $W=\bigcup\limits_{i=1}^s W_i=(\bigcup\limits_{i=1}^s W_i')\cup (\bigcup\limits_{i=1}^s W_i'')$.
  Define $W'=\bigcup\limits_{|W'_i|\ge \binom{s}{i}|F|}W'_i$ and $W''=W\backslash W'$.
  Then, by \cref{cl:Wi'} (b), we have
\begin{align}\label{eq:W''}
  |W''|&= \sum\limits_{|W'_i|< \binom{s}{i}|F|}|W'_i|+\sum\limits_{i=1}^s|W_i''|\nonumber\\[2mm]
  &< \sum\limits_{i=1}^s\binom{s}{i}|F|+\sum\limits_{i=1}^s\binom{s}{i}|F|\nonumber\\[2mm]
  &<2^{s+1}|F|.
\end{align}

   To prove \cref{thm:stability} (i), let $$c_1:=\ex(s, K_{r-1},\mathcal{D}(F))-\ex(s-1, K_{r-1},\mathcal{D}(F))-z>0.$$

   The next claim shows that $|W_s|$ is not too small.

\begin{claim}\label{cl:W_s}
  $|W_s| \geq \max\{|F|,s+2\}$.
\end{claim}

\begin{pf}
   If $|W_s|< \max\{|F|,s+2\}<\binom{s}{i}|F|$, then $W_s\subseteq W''$, which implies that each vertex in $W'$ has degree at most $s-1$.
  Combining this with \cref{cl:Wi'} (i), for any $v\in W'$, the number of $r$-cliques containing $v$ is $\mathcal{N}(K_{r-1},G[N(v)])\leq \mathrm{ex}(s-1,K_{r-1},\mathcal{D}(F))$.
  Note that each vertex in $W''\subseteq W$ has degree at most $s$. Thus
\begin{align*}
   \mathcal{N}(K_r,G)
  &\leq |W'|\ex(s-1,K_{r-1},\mathcal{D}(F))+|W''|\binom{s}{r-1}+\binom{n-|W|}{r}\\
  &< n\cdot(\ex(s,K_{r-1},\mathcal{D}(F)-z-c_1)+2^{s+1}|F|\binom{s}{r-1}+\binom{4s}{r} \quad \text{(by \eqref{eq:W} and \eqref{eq:W''})}\\
  &<(n-s)\ex(s,K_{r-1},\mathcal{D}(F))- z n,
\end{align*}
when $n>(2^{s+1}|F|\binom{s}{r-1}+\binom{4s}{r}+s\cdot\ex(s,K_{r-1},\mathcal{D}(F)))/c_1$, leading to a contradiction.
\end{pf}

   By \cref{cl:W_s}, $|W_s| \geq \max\{|F|,s+2\}$, which implies that $|B| = s$ and the vertices in $B$ have at least $s+2$ common neighbours.
  If there exists an edge in $G-B$, we can greedily find an $(s+1)$-matching as each vertex in $B$ has degree at least $s+2$.
  Thus, $V(G)\backslash B$ forms an independent set of size $n-s$.
  This concludes the proof of \cref{thm:stability} (i).

  Next, we prove \cref{thm:stability} (ii).
  If $\ex(s, K_{r-1},\mathcal{D}(F))-\mathcal{N}(K_{r-1},H)> z$ holds for all $F$-free graphs $H$ with $s$ vertices that are not in $\mathrm{EX}(s, K_{r-1},\mathcal{D}(F))$, then let $$c_2:=\ex(s, K_{r-1},\mathcal{D}(F))-\max\limits_{\substack{F\not\subseteq H,~|H|=s\text{ and} \\
\text{$H\notin\mathrm{EX}(s, K_{r-1},\mathcal{D}(F))$}}}\{\mathcal{N}(K_{r-1},H)\}-z>0.$$

  Since $B$ has $|W_s|\ge |F|$ common neighbours, $G[B]$ must be $\mathcal{D}(F)$-free.
  Then, we will prove that $G[B]\in \mathrm{EX}(s, K_{r-1},\mathcal{D}(F))$.
  If $G[B]\notin \mathrm{EX}(s, K_{r-1},\mathcal{D}(F))$, then
\begin{align*}
   \mathcal{N}(K_r,G)
   & \leq |W|\mathcal{N}(K_{r-1},G[B])+\mathcal{N}(K_{r},G[B])\\
  &\leq (n-s)(\mathrm{ex}(s,K_{r-1},\mathcal{D}(F))-z-c_2)+\binom{s}{r}\\
  &<(n-s)\mathrm{ex}(s,K_{r-1},\mathcal{D}(F))-z n
\end{align*}
 when $n>(\binom{s}{r}+sz+sc_2)/c_2$, resulting in a contradiction.
 Set $D=G[B]\in \mathrm{EX}(s, K_{r-1},\mathcal{D}(F))$.
 Recall that $V(G) \setminus B$ is an independent set of size $n-s$.
 So $G$ is a subgraph of $D \vee I_{n-s}$, we complete the proof of \cref{thm:stability}  (ii).
\end{proof}

 Using \cref{thm:stability}, we can prove \cref{cor1}.

\begin{proof}[\textbf{Proof of \cref{cor1}}]

  Assume $G$ is an $\{F, M_{s+1}\}$-free graph with $\mathcal{N}(K_r, G)=\ex(n, K_r,\{F, M_{s+1}\})$.
  It is easily seen that $D_F(s, r) \vee I_{n-s}$ is $\{F, M_{s+1}\}$-free. So
\begin{align}\label{eq:444}
  \ex(n, K_r,\{F, M_{s+1}\}) \geq \mathcal{N}(K_r, D_F(s, r) \vee I_{n-s}).
\end{align}
  On the other hand, the assumption that $\ex(s, K_{r-1}, \mathcal{D}(F))>\ex(s-1, K_{r-1}, \mathcal{D}(F))$ implies that
$$
\ex(s, K_{r-1}, \mathcal{D}(F))-\ex(s-1, K_{r-1}, \mathcal{D}(F))>1 / 2.
$$
Obviously, for all $\mathcal{D}(F)$-free graphs $H$ with $s$ vertices that are not in $\mathrm{EX}(s, K_{r-1}, \mathcal{D}(F))$, we have
$$
\ex(s, K_{r-1}, \mathcal{D}(F))-\mathcal{N}\left(K_{r-1}, H\right)>1 / 2.
$$
By \cref{thm:stability}, $G$ is a subgraph of $D \vee I_{n-s}$ for some $D \in \mathrm{EX}(s, K_{r-1}, \mathcal{D}(F))$. So
\begin{align}\label{eq:555}
\ex(n, K_r,\{F, M_{s+1}\}) \leq \mathcal{N}(K_r, D_F(s, r) \vee I_{n-s}).
\end{align}
Combining with \eqref{eq:444} and \eqref{eq:555}, we get
$$
\ex(n, K_r,\{F, M_{s+1}\})=\mathcal{N}(K_r, D_F(s, r) \vee I_{n-s}).
$$
\end{proof}

\begin{proposition}\label{lmm:3.2}
   Let $r\ge 3$ be an integer.
    Let $\mathcal{F}=\{F_1, \ldots, F_k\}$, where $\chi(F_1)=\cdots=\chi(F_k)\ge r$.
    Suppose $D$ is an $\mathcal{F}$-free graph on $n$ vertices such that $\mathcal{N}(K_{r-1},D)=\ex(n,K_{r-1},\mathcal{F})$.
    Let $t=|F_1|+\cdots+|F_k|$, then every vertex in $D$ is contained in at least $t^{r-2}$ copies of $K_{r-1}$.
\end{proposition}

\begin{proof}
    Since $\chi(F_i)\ge r$ for each $i\in[k]$, we have $$\mathcal{N}(K_{r-1},D)=\ex(n,K_{r-1},\mathcal{F})\ge\mathcal{N}(K_{r-1},T_{r-1}(n))=\Theta(n^{r-1}).$$
   Then $D$ must contain a copy of balanced complete $(r-1)$-partite graph $K_{t,t,\ldots,t}$, denoted by $T$, as a subgraph.
   Otherwise, $D$ is $K_{t,t,\ldots,t}$-free.
   By \cref{prop:erdos-stone}, $D$ has $o(n^{r-1})$ copies of $K_{r-1}$, a contradiction.

   Let $\mathcal{N}_v(K_{r-1}, D)$ be the number of $(r-1)$-cliques in $D$ containing $v$.
   Suppose to the contrary that there exists a vertex $v$ with $\mathcal{N}_v(K_{r-1}, D)<t^{r-2}$.
   Clearly, there is at least one vertex $u\in V(T)$ that is not adjacent to $v$.
   Otherwise, $\mathcal{N}_v(K_{r-1}, D)\ge \mathcal{N}(K_{r-2},T)>t^{r-2}$.
   We remove all edges incident to $v$ and add the edge set $E=\{vw:w\in N_T(u)\}$.
   Let $D'$ denote the resulting graph after these changes.
   We claim that $D'$ is also $\mathcal{F}$-free.
   Indeed, if $F_i\in\mathcal{F}$ appears, then $F_i$ must contain $v$.
   Since $t > |F_i| - 1$, we can find $v'$ in the same colour class as $u$ in $T$ to replace $v$, thus obtaining $F_i$ in $D'-v$.
   However, $D'-v$ is a subgraph of $D$, which contradicts the assumption that $D$ is $\mathcal{F}$-free.
   We removed $\mathcal{N}_v(K_{r-1}, D)$ copies of $K_{r-1}$, but added at least $t^{r-2}$ copies of $K_{r-1}$, then $\mathcal{N}(K_{r-1}, D')>\mathcal{N}(K_{r-1}, D)$ contradicting the maximality of $D$.
   Thus, every vertex in $D$ is contained in at least $t^{r-2}$ copies of $K_{r-1}$.
\end{proof}

\begin{proof}[\textbf{Proof of \cref{cor2}}]
  If $\chi(F)>r \geq 3$, then $\chi(F') \geq r$ for any graph $F' \in \mathcal{D}(F)$.
  We claim that $$\ex(s, K_{r-1},\mathcal{D}(F))-\ex(s-1, K_{r-1},\mathcal{D}(F))>0$$ when $s$ is sufficiently large.
  Otherwise,  $\ex(s, K_{r-1},$ $\mathcal{D}(F))-\ex(s-1, K_{r-1},\mathcal{D}(F))=0,$ then there exists an extremal graph $D\in \mathrm{EX}(s, K_{r-1},\mathcal{D}(F))$ and a vertex $v\in V(D)$ such that $v$ is not contained in any $r$-clique of $D$, which  contradicts \cref{lmm:3.2}.
   Hence, $\ex(s, K_{r-1},\mathcal{D}(F))-\ex(s-1, K_{r-1},\mathcal{D}(F))>0$.
   It follows from \cref{cor1} that $$\ex(n, K_r,\{F, M_{s+1}\})=\mathcal{N}(K_r,D_F(s,r)\vee I_{n-s}).$$
\end{proof}

\section{Forbidding bipartite graphs}\label{section:bipartite}

\subsection{Stability result}

  Let $F$ be a bipartite graph.
  Recall that $p = p(F)$ represents the smallest feasible order of a colour class in a proper two-colouring of $F$, and $q=|F|-p$.
  For cases where $p \leq s$, the subsequent lemma establishes an approximate structure for an extremal graph in $\mathrm{EX}(n, K_r, \{F, M_{s+1}\})$.

\begin{lemma}\label{lmm:stability-2}
  Let $s$ be an integer and $F$ be a bipartite with $p=p(F) \leq s$.
  For any sufficiently large $n$, there exist an integer $t\leq \binom{2s}{p}(q-1)+2s+1-p$ and an $\{F, M_{s+1}\}$-free graph $H$ on $n$ vertices with $\mathcal{N}(K_r,H)=\ex(n,K_r,\{F, M_{s+1}\})$ and a partition $V(H)=X\cup Y\cup Z$ that satisfies the following:
\begin{enumerate}[label=(\arabic*)]
    \item[\rm (i)] $H[X]=K_{p-1}$;
    \item[\rm (ii)] $Y$ is an independent set with $|Y|=n-t-p+1$ and every vertex in $Y$ has the neighbourhood $X$;
    \item[\rm (iii)] every vertex in $Z$ has degree at least $p$.
\end{enumerate}
\end{lemma}

\begin{proof}[\textbf{Proof of \cref{lmm:stability-2}}]
  Let $G$ be an $\{F, M_{s+1}\}$-free graph with $\mathcal{N}(K_r,G)=\ex(n,K_r,\{F, M_{s+1}\})$.
  Note that $K_{p-1} \vee I_{n-p+1}$ is $\{F, M_{s+1}\}$-free, leading to:
\begin{align}\label{eq:p-1}
  \mathcal{N}(K_r,G)\geq \mathcal{N}(K_r,K_{p-1}\vee I_{n-p+1})=\binom{p-1}{r-1}(n-p+1)+\binom{p-1}{r}.
\end{align}

   Let $U$ denote the set of vertices that are matched under a maximum matching of $G$.
  Then $|U| \leq 2s$, and $V(G) \setminus U$ forms an independent set.
  Let $L:=\{v\in V(G)\setminus U:d(v)\ge p\}$ be the set of vertices outside $U$ with degrees at least $p$.
  We first estimate the size of $L$.

\begin{claim}
   $|L|\le \binom{2s}{p}(q-1)$.
\end{claim}

\begin{pf}
  Suppose, for the sake of contradiction, that $|L| \ge \binom{2s}{p}(q-1)+1$.
  Assume, without loss of generality, that each vertex in $L$ has degree exactly $p$ (we can remove some edges to achieve this).
 Each vertex in $L$ corresponds to a $p$-subset of $U$ as its neighbourhood.
 Since $U$ has at most $\binom{2s}{p}$ $p$-subsets, by the Pigeonhole principle, there exist $q$ vertices in $L$ sharing $p$ common neighbours in $U$.
 Consequently, $G$ contains a $K_{p,q}$ and, therefore, contains an $F$, which leads to a contradiction.
\end{pf}

  Let $W=V(G)\setminus(U\cup L)$.
  Note every vertex in $W$ has degree at most $p-1$ and
\begin{align}\label{eq:777}
    |W| \geq n - 2s - \binom{2s}{p}(q-1).
\end{align}

\begin{claim}\label{cl:common-neighbour}
  There exists a subset $X\subseteq U$ such that
  \begin{align*}
   G[X]=K_{p-1} \quad \text{ and }\quad |N(X)|\ge \max\{|F|+1,2s+3\}.
\end{align*}
\end{claim}

\begin{pf}
  Let $W'\subseteq W$ be the set of vertices whose neighbourhood induces a $(p-1)$-clique.
  Then, for any vertex $v\in W\setminus W'$, the number of $r$-cliques containing $v$ is at most $\binom{p-1}{r-1}-1$.

  If $|W'|< \max\{|F|+1,2s+3\}\binom{2s}{p-1}$, then
\begin{align*}
  \mathcal{N}(K_r,G)
   & \leq |W'|\binom{p-1}{r-1}+|W\setminus W'|\left(\binom{p-1}{r-1}-1\right)+\binom{n-|W|}{r}\\[2mm]
   & < \max\{|F|+1,2s+3\}\binom{2s}{p-1}\binom{p-1}{r-1}+|W|\left(\binom{p-1}{r-1}-1\right)+\binom{n-|W|}{r}\\[2mm]
   & < \max\{|F|+1,2s+3\}\binom{2s}{p-1}\binom{p-1}{r-1}+n\left(\binom{p-1}{r-1}-1\right)+\binom{2s+\binom{2s}{p}(q-1)}{r}\quad \text{(by \eqref{eq:777})}\\[2mm]
   & = n\binom{p-1}{r-1}-\left(n-\max\{|F|+1,2s+3\}\binom{2s}{p-1}\binom{p-1}{r-1}-\binom{2s+\binom{2s}{p}(q-1)}{r}\right)\\[2mm]
   & < n\binom{p-1}{r-1}-\binom{p-1}{r-1}(p-1)+\binom{p-1}{r}\quad \text{(as $n$ is sufficiently large)}\\[2mm]
   & = \binom{p-1}{r-1}(n-p+1)+\binom{p-1}{r},
\end{align*}
 which contradicts \eqref{eq:p-1}.

  Therefore, $|W'| \geq \max\{|F| + 1, 2s + 3\}\binom{2s}{p-1}$.
  Since the number of $(p-1)$-sets in $U$ is at most $\binom{2s}{p-1}$, by double counting, there is at least one $(p-1)$-set, denoted as $X$, which is the neighbourhood of a vertex in $W'$ and has at least $\max\{|F|+1,2s+3\}$ common neighbours.
  That is, $|N(X)|\ge \max\{|F|+1,2s+3\}$.
  Moreover, by the definition of $W'$, $G[X]=K_{p-1}$.
\end{pf}

  If there exists a vertex $v$ in $V(G)$ with a degree at most $p-1$, we transform its neighbourhood into $X$, creating the graph $G'$.
  We claim that $G'$ is $\{F, M_{s+1}\}$-free.
  If $G'$ contains a copy of $F$ or $M_{s+1}$, referred to as $Q$, then $Q$ must include the vertex $v$.
  By \cref{cl:common-neighbour}, $|N(X)|\ge \max\{|F|+1,2s+3\}$.
  Therefore, there exists a vertex $v' \in N(X) \setminus V(Q)$ such that its neighbourhood is $X$.
  Then we replace vertex $v$ with $v'$.
  Since $v$ and $v'$ have identical neighbourhoods, this implies that there is a copy of $F$ or $M_{s+1}$ within $G' - v$.
  However, $G' - v$ is a subgraph of $G$, leading to a contradiction.
  Importantly, the number of $r$-cliques in $G'$ is at least as many as in $G$.

  We keep repeating this process until no vertex in the graph has a degree less than $p-1$, and all vertices with degree $p-1$ share the same set of neighbours.
  The resultant graph is denoted as $H$.
  Since $H$ is $\{F, M_{s+1}\}$-free and $\mathcal{N}(K_r,H)\ge \mathcal{N}(K_r,G)$, $H$ is an extremal graph of $\{F, M_{s+1}\}$, that is, $\mathcal{N}(K_r,H)=\ex(n,K_r,\{F, M_{s+1}\})$.
  Note that $H[X]=K_{p-1}$.
  Let $Y$ denote the set of vertices in $H$ with a degree of $p-1$, and let $t = n - p + 1 - |Y|$.
  It is worth noting that $t \leq 2s + \binom{2s}{p}(q-1) + 1 - p$, since $|Y|\geq |W|\geq n-2s-\binom{2s}{p}(q-1)$.
  Let $Z=V(H)\setminus (X\cup Y)$.
  Then every vertex in $Z$ has a degree of at least $p$.
\end{proof}

\begin{proof}[\textbf{Proof of \cref{thm:thm2.3}}]
As graph $K_{p-1}\vee I_{n-p+1}$ is $\{F, M_{s+1}\}$-free, we have
$$\ex(n, K_r,\{F, M_{s+1}\})\geq\binom{p-1}{r}+\binom{p-1}{r-1}(n-p+1).$$
On the other hand, by Lemma \ref{lmm:stability-2}, we have
$$\ex(n, K_r,\{F, M_{s+1}\})\le \binom{t+p-1}{r}+\binom{p-1}{r-1}(n-t-p+1).$$
\end{proof}

\subsection{Forbidding even paths}





  We construct $G_1$ and $G_2$ as follows, which provide the lower bounds of $\ex(n,K_r,\{P_{2p},M_{s+1}\})$.
  It is easy to check that $G_1$ and $G_2$ are $\{P_{2p},M_{s+1}\}$-free.
\begin{construction}
  Let $s$, $p$ and $n$ be three integers.
  Let $a=\lfloor\frac{s-p+1}{p-1}\rfloor$ and $b=s-(a+1)(p-1)$.
  Define
\begin{itemize}
\item $G_1=(K_{p-1}\vee I_{n-p+1-a(2p-1)})\cup a K_{2p-1}$
\item $G_2=(K_{p-1}\vee I_{n-p+1-a(2p-1)-2b-1})\cup a K_{2p-1}\cup K_{2b+1}$
\end{itemize}
\end{construction}

  Before proceeding with the proof, we provide two useful propositions.

\begin{proposition}\label{proposition:1}
  Let $G$ be a graph, and $C$ be a connected component of $G$.
  Define $G' = (G-V(C)) \cup K_{|C|}$, which represents replacing $C$ in $G$ with a complete graph of the same order.
  If $|C|$ contains a path of length $|C|$, then $\nu(G') = \nu(G)$.
\end{proposition}

\begin{proof}
  Since $C$ contains a path of length $|C|$, we have $\nu(C) = \lfloor |C|/2 \rfloor$.
  Clearly, $\nu(K_{|C|}) = \lfloor |C|/2 \rfloor = \nu(C)$.
  Therefore, we have $\nu(G') = \nu(K_{|C|}) + \nu(G' - V(K_{|C|})) = \nu(C) + \nu(G-V(C)) = \nu(G)$.
\end{proof}

\begin{proposition}\label{proposition:2}
 Let $G$ be a graph, $C$ be a connected component of $G$, and $w$ be a vertex in $V(G)\setminus V(C)$.
 Define the graph $G'$ as follows:
\begin{align*}
   V(G')=V(G) \quad \text{and} \quad  E(G')=E(G)-\{wv:v\in N_G(w)\}+\{wv:v\in V(C)\}.
\end{align*}
  If $|C|$ is even and $\nu(C) = |C|/2$, then $\nu(G') \leq \nu(G)$.
\end{proposition}

\begin{proof}
  Let $C' = G'[V(C) \cup {w}]$. Then $C'$ is a connected component of $G'$.
  Since $|C|$ is even, $\nu(C') = |C|/2$.
  Therefore, $\nu(G') = \nu(C') + \nu(G' - V(C')) \leq \nu(C) + \nu(G - V(C)) = \nu(G)$.
\end{proof}

\begin{proof}[\textbf{Proof of \cref{2l}.}]
  The lower bounds are established by considering graphs $G_1$ and $G_2$.
  For the upper bound, let $G$ be an extremal graph with the properties of \cref{lmm:stability-2}.
  Additionally, we assume that $G$ has the maximum number of edges among all extremal graphs with $n$ vertices. (The operation mentioned in \cref{lmm:stability-2} does not decrease the number of edges.)
  We will show that $G$ is isomorphic to $G_1$ or $G_2$.

  By \cref{lmm:stability-2}, there is a constant $t$ and a partition $V(G)=X\cup Y\cup Z$ satisfying the conclusions of \cref{lmm:stability-2}.
  Suppose that $X=\{v_1,\ldots,v_{p-1}\}$.
  We first prove that $X$ and $Z$ are not connected.

\begin{claim}\label{cl:disconnected}
   There are no edges between $X$ and $Z$.
\end{claim}

\begin{pf}
  Without loss of generality, assume that there is an edge  $uv_1\in E(G)$ for some $u\in Z$.
  It follows from \cref{lmm:stability-2} (iii) that $d(u)\geq p$.
  By \cref{lmm:stability-2} (i), $|X|=p-1$, thus $u$ has a neighbour in $Z$, denoted as $u'$.
  Moreover, all vertices in $Y$ have the neighbourhood $X$ by \cref{lmm:stability-2} (ii).
  Hence, there exist $w_1,\ldots,w_{p-1}\in Y$ such that $P=u'uv_1w_1\cdots v_{p-1}w_{p-1}$ is a path of length $2p$.
  However, this contradicts the assumption that $G$ is $P_{2p}$-free.
\end{pf}


  Combining \cref{cl:disconnected} and \cref{lmm:stability-2} (ii), there are no edges between $Z$ and $X\cup Y$.
  We then proceed to characterize the structure of the subgraphs induced by $Z$.

\begin{claim}\label{cl:claim3.7}
  For any component $C$ in $G[Z]$, $C$ forms an odd clique with $p+1\leq |C|\leq 2p-1$.
\end{claim}

\begin{pf}
  Consider a component $C$ of $G[Z]$.
  Since $Z$ and $X \cup Y$ are not connected, we have $\delta(C) \geq \delta(G) \geq p$, implying $|C| \geq p+1$.
  Let $P$ be the longest path in $C$.
  By utilizing \cref{Dirac} and noting that $C$ is $P_{2p}$-free, we deduce that $|C|\ge|P|\geq \min \{|C|, 2p+1\}$.
  This leads to $p+1\leq|P|=|C|\leq 2p-1$.

  If $|C|$ is even, then $C$ contains a path of length $|C|$ and consequently contains a $|C|/2$-matching.
  According to \cref{proposition:1}, replacing $C$ with a complete graph of the same order does not increase the matching number.
 Since $|C| \leq 2p-2$, it also does not create $P_{2p}$.
 Considering that $G$ has the maximum number of edges among all extremal graphs, $C$ must be a clique.
 Now we transform the neighbourhood of a vertex $w$ in $Y$ to $C$, resulting in the graph $G'$.
 By \cref{proposition:2}, $\nu(G') \leq \nu(G)$.
 Note that $|C| \leq 2p-2$, so $G'$ does not contain $P_{2p}$.
 During this process, we remove $\binom{p-1}{r-1}$ $r$-cliques and add $\binom{|C|}{r-1}$ $r$-cliques.
 Since $|C| \geq p+1$, we have $\binom{|C|}{r-1} > \binom{p-1}{r-1}$, which contradicts the maximality of $G$.
 Therefore, $|C|$ is odd. Again, by utilizing \cref{proposition:1} and considering the maximality of the edge number, we conclude that $C$ is a clique.
\end{pf}

\begin{claim}\label{cl:claim3.8}
  There is at most one component in $G[Z]$ with an order smaller than $2p-1$.
\end{claim}

\begin{pf}
  By \cref{cl:claim3.7}, we know that the connected components of $G[Z]$ are all cliques of odd order.
  We proceed by contradiction.
  Suppose $K_x$ and $K_y$ are two components (thus are two cliques) of $G[Z]$ with sizes $x$ and $y$, respectively, where $p+1 \leq y \leq x \leq 2p-3$.
  We replace $K_x \cup K_y$ with $K_{x+2} \cup K_{y-2}$. It is easy to verify that this operation does not increase $\nu(G)$ and does not create $P_{2p}$. According to \cref{2021C}, we have
\begin{align*}
  \binom{x+2}{r}+\binom{y-2}{r}>\binom{x}{r}+\binom{y}{r}.
\end{align*}
 This means the resulting graph has more $r$-cliques than $G$, contradicting the maximality of $G$.
\end{pf}

  Consider the $2p-1$ vertices in an independent $(2p-1)$-clique and the $2p-1$ vertices in $Y$.
  The former can contribute a total of $\binom{2p-1}{r}$ copies of $K_r$, while the latter can contribute $(2p-1)\binom{p-1}{r-1}$ copies of $K_r$.
  By \cref{lmm:Vandermonde}, we have
\begin{align*}
  \binom{2p-1}{r}
  &=\sum_{i=0}^r\binom{p-1}{i}\binom{p}{r-i}\\
  &>\binom{p-1}{1}\binom{p}{r-1}+\binom{p-1}{r-1}\binom{p}{1}\\
  &>(2p-1)\binom{p-1}{r-1}.
\end{align*}
  Therefore, it is desirable to maximize the number of independent $(2p-1)$-cliques in $G[Z]$, while ensuring that the matching number does not exceed $s$.
  Consequently, there exist $a:=\lfloor\frac{s-p+1}{p-1}\rfloor$ independent $(2p-1)$-cliques within $Z$.

  If there is no component of order smaller than $2p-1$ in $G[Z]$, then $G[Z]=a K_{2p-1}$.
  Thus
  $$G=(K_{p-1}\vee I_{n-p+1-a(2p-1)})\cup a K_{2p-1}=G_1.$$

  If there is a component, denoted as $C'$, in $G[Z]$ of order smaller than $2p-1$,
  then by \cref{cl:claim3.7} and \cref{cl:claim3.8}, $C'$ is a clique with $p+1\le|C'|\leq 2p-3$.
  We claim that $\nu(C')=s-(a+1)(p-1)$.
  Since, by \cref{lmm:stability-2} (ii), $G[X,Y]=K_{p-1,n-t-p+1}$ and $n$ is sufficiently large, we have $\nu(G[Z])\le s-p+1$.
  If $\nu(G[Z])< s-p+1$, we alter the neighbourhood of a vertex $w\in Y$ to $C'$.
  Notice that this operation does not create a copy of $P_{2p}$.
  We remove $\binom{p-1}{r-1}$ $r$-cliques and add $\binom{|C'|}{r-1}$ $r$-cliques.
  Note that $|C'|>p-1$, this contradicts the maximality of $G$.
  Thus $\nu(C')=s-(p-1)-a(p-1)=s-(a+1)(p-1)$.

  Since $C'$ is an odd clique and $\nu(C')=s-(a+1)(p-1)$, $C'=K_{2b+1}$.
  Therefore, $$G=(K_{p-1}\vee I_{n-p+1-a(2p-1)-2b+1})\cup a K_{2p-1}\cup K_{2b+1}=G_2.$$
\end{proof}

\subsection{Forbidding odd paths}

  We construct graphs $G_3$, $G_4$, $G_5$ and $G_6$ as follows, which provide the lower bounds of $\ex(n,K_r,\{P_{2p+1},M_{s+1}\})$.

\begin{construction}
  Let $q=\lfloor\frac{s-p+1}{p-1}\rfloor$ and $t=s-(q+1)(p-1)$.
 Define
\begin{eqnarray*}
	G_3=\left\{
	\begin{array}{ll}
		(K_{p-1}\vee I_{n-p-q(2p-1)})\cup q K_{2p-1} &\mbox{if $t=0$}\\
		(K_{p-1}\vee I_{n-p-q(2p-1)-2t})\cup q K_{2p-1}\cup K_{2t+1} &\mbox{if $t\not=0$}.
	\end{array}
	\right.
\end{eqnarray*}

  Let $c=\lfloor\frac{s-p+1}{p}\rfloor$ and $d=s-p+1-cp$. Define
\begin{eqnarray*}
	G_4=\left\{
\begin{array}{ll}
	(K_{p-1}\vee I_{n-p+1-2cp})\cup c K_{2p} &\mbox{if $d=0$}\\
	(K_{p-1}\vee I_{n-p-2cp-2d})\cup c K_{2p}\cup K_{2d+1} &\mbox{if $d\not=0$}.
\end{array}
\right.
\end{eqnarray*}

\begin{itemize}
\item
For $1\leq d\leq p-2$, we define
$$G_5=(K_{p-1}\vee I_{n-p+1-2p(c+d+1-p)-(p-d)(2p-1)})\cup (c+d+1-p) K_{2p}\cup (p-d)K_{2p-1}.$$
\item
For $1\leq d\leq p-1$, we define $G_6=(K_{p-1}\vee (K_2\cup I_{n-p-1-2cp}))\cup c K_{2p}.$
\end{itemize}
\end{construction}
\noindent
It is easy to check that $G_3$, $G_4$, $G_5$ and $G_6$ are $\{P_{2p+1},M_{s+1}\}$-free.
\begin{proof}[\textbf{Proof of Theorem \ref{2l+1}.}]
  The lower bounds are established by graphs $G_3$, $G_4$, $G_5$,  and $G_6$.
  For the upper bounds, consider $G$ as an extremal graph guaranteed by \cref{lmm:stability-2}.
  Additionally, we assume that $G$ has the maximum number of edges among all extremal graphs with $n$ vertices.
  We will show that $G$ is isomorphic to  $G_3$, $G_4$, $G_5$, or $G_6$.

By \cref{lmm:stability-2}, there is a constant $t$ and a partition $V(G)=X\cup Y\cup Z$ satisfies the conclusions of \cref{lmm:stability-2}.
Suppose that $X=\{v_1,\ldots,v_{p-1}\}$.

  We are considering two cases regarding the existence or non-existence of an edge between $X$ and $Z$, respectively.

\begin{mycase}{Case 1.}
   $e(G[X,Z])=0$.
\end{mycase}

  Consider a component $C$ of $G[Z]$.
  Since $\delta(C)\geq p$, we have $|C|\geq p+1$.
  Let $P$ be a longest path of $C$.
  Since $C$ is $P_{2p+1}$-free, according to \cref{Dirac}, we have $|C|\ge |P|\geq \min \{|C|, 2p+1\}$.
  Therefore, $p+1\leq|P|=|C|\leq 2p$.

  We now claim that every component of $G[Z]$ is a clique.
  Since $C$ contains a copy of $P_{|C|}$ and thus contains a $\lfloor|C| / 2\rfloor$-matching.
  According to \cref{proposition:1}, replacing $C$ with a complete graph of the same order does not increase the matching number.
  Since $|C| \leq 2 p$, it also does not create $P_{2 p+1}$.
  Then $C$ must be a clique as $G$ has the maximum number of edges among all extremal graphs.

  The following four claims give a more precise characterization of the components of $G[Z]$.

\begin{claim}\label{cl:even}
  Each even component (if exists) in $G[Z]$ is a clique with order $2p$.
\end{claim}

\begin{pf}
  Suppose there exists an even component $C$ with order at most $2p-2$,
  we can modify $G$ by changing the neighbourhood of a vertex $w\in Y$ to $C$.
  According to \cref{proposition:2}, this operation does not increase the matching number.
  Moreover, since $|C|\le 2p-2$, it also does not create a copy of $P_{2p+1}$.
  We remove $\binom{p-1}{r-1}$ $r$-cliques and add $\binom{|C|}{r-1}$ $r$-cliques.
  Since $|C|> p-1$, this contradicts the maximality of $G$.
\end{pf}

\begin{claim}\label{claim5}
  In $G[Z]$, there exists at most one odd component with an order smaller than $2p-1$.
  Moreover, if there are multiple odd components and at least one even component, then all odd components have an order of $2p-1$.
\end{claim}

\begin{pf}
   Assume that $K_x$ and $K_y$ are two odd components of $G[Z]$ with $p+1\leq y\leq x\leq 2p-3$.
   We can replace $K_x\cup K_y$ with $K_{x+2}\cup K_{y-2}$ without increasing $\nu(G)$ or creating a copy of $P_{2p+1}$.
   However, by \cref{2021C}, we have $\binom{x+2}{r}+\binom{y-2}{r}>\binom{x}{r}+\binom{y}{r}$, which implies that the resulting graph has more $r$-cliques than $G$, contradicting the maximality of $G$.

   Suppose there are multiple odd components and at least one even component. If the smallest odd component has an order of $2t-1$, where $(p+2)/2\leq t\leq p-1$,
  we claim that
\begin{equation}
  \binom{2p}{r}+\binom{2t-1}{r}+\binom{p-1}{r-1}\geq\binom{2p-1}{r}+\binom{2t+1}{r}, \label{1}
\end{equation}
and
\begin{equation}
\binom{2p-1}{r}+\binom{2t-1}{r}\geq \binom{2p}{r}+\binom{2t-3}{r}+\binom{p-1}{r-1}. \label{2}
\end{equation}
   Otherwise,
\begin{itemize}
    \item if $\binom{2p}{r}+\binom{2t-1}{r}+\binom{p-1}{r-1}<\binom{2p-1}{r}+\binom{2t+1}{r}$, then we remove a vertex from $Y$ and replace $K_{2p} \cup K_{2t-1}$ with $K_{2p-1}\cup K_{2t+1}$;
    \item if $\binom{2p-1}{r}+\binom{2t-1}{r}<\binom{2p}{r}+\binom{2t-3}{r}+\binom{p-1}{r-1}$, then we add a vertex to $Y$ and replace $K_{2p-1} \cup K_{2t-1}$ with $K_{2p}\cup K_{2t-3}$.
\end{itemize}
   Neither operation increases $\nu(G)$ or creates a copy of $P_{2p+1}$, but they do increase the number of $r$-cliques, which contradicts the maximality of $G$.

   Combining (\ref{1}) and (\ref{2}), we obtain $2\binom{2t-1}{r}\geq \binom{2t-3}{r}+\binom{2t+1}{r}$, which contradicts $\binom{2t-1}{r}+\binom{2t-1}{r}< \binom{2t-3}{r}+\binom{2t+1}{r}$ from  \cref{2021C}.
\end{pf}


\begin{claim}\label{cl:2p-1}
    There are at most $p-1$ independent $(2p-1)$-cliques in $G[Z]$.
\end{claim}

\begin{pf}
  Suppose to the contrary that there exist $p$ disjoint copies of  $K_{2p-1}$ in $G[Z]$.
   We replace the copy of $p K_{2p-1}$ in $G[Z]$ with $(p-1)K_{2p}$ and add $p$ isolated vertices.
  Let $G'$ be the resulting graph. Then $|G'|=|G|$.
  During the transition from $G$ to $G'$, we remove $p\binom{2p-1}{r}$ $r$-cliques and add $(p-1)\binom{2p}{r}$ $r$-cliques.
  Note that for $r\ge 3$, we have
\begin{align*}
  \frac{(p-1)\binom{2p}{r}}{p\binom{2p-1}{r}}  = \frac{\frac{(p-1)\cdot(2p)!}{r!(2p-r)!}}{\frac{p\cdot(2p-1)!}{r!(2p-1-r)!}}=\frac{2p-2}{2p-r}>1.
\end{align*}
  This contradicts the maximality of $G$.
\end{pf}

 Next claim shows that the matching number of $G[Z]$ is exactly $s-p+1$.

\begin{claim}\label{cl:matching number of Z}
    $\nu(G[Z])=s-p+1$.
\end{claim}

\begin{pf}
  By \cref{lmm:stability-2}, we know that $G[X\cup Y]=K_{p-1}\vee I_{n-t-p+1}$.
  If $\nu(G[Z])<s-p+1$, then adding an edge in $Y$ increases the number of edges by 1.
  Since there are no edges between $Z$ and $X\cup Y$, it is easy to verify that the resulting graph still does not contain $P_{2p+1}$ or $M_{s+1}$, which contradicts the maximality of edges in $G$.
\end{pf}

We claim that under the condition $\nu(G)\le s$,  the number of independent $2p$-cliques in the extremal graph should be as large as possible.
   Note that if $2p$ vertices are in $K_{2p}$ can contribute $\binom{2p}{r}$ $r$-cliques, while in $Y$ they can only contribute $2p\binom{p-1}{r-1}$ $r$-cliques.
  By \cref{lmm:Vandermonde}, we have
\begin{align*}
  \binom{2p}{r}
  =\sum_{i=0}^r\binom{p}{i}\binom{p}{r-i}
  >\binom{p}{1}\binom{p}{r-1}+\binom{p}{r-1}\binom{p}{1}
  > 2p\binom{p-1}{r-1}.
\end{align*}

   If we consider the case that $2p-1$ vertices are respectively in $K_{2p-1}$ and in $Y$, we can also obtain similar results.
   Therefore, it is desirable to have as many independent $2p$-cliques or $(2p-1)$-cliques in $G$ as possible.

  If there are no even components in $G[Z]$, then by \cref{cl:even} and \cref{cl:2p-1}, $G[Z]=q K_{2p-1}$ if $t=0$ and  $G[Z]=q K_{2p-1}\cup K_{2t+1}$ if $t\not=0$.
  This implies that $G=G_3.$

  Suppose that there is at least one even components in $G[Z]$.
  Then by \cref{cl:even}--\cref{cl:matching number of Z}, we know that $G[Z]$ is either
   $c K_{2p}$ if $d=0$ and
  $c K_{2p}\cup K_{2d+1}$ with $d+cp=s-p+1$ if $d\not=0$
   or $x K_{2p}\cup yK_{2p-1}$ for some $x\ge 1$ and $p-1\geq y\geq 2$ with $xp+y(p-1)=s-p+1$.

   For the former, we have $$G=G_4.$$

   For the later, we have $G[Z]=x K_{2p}\cup yK_{2p-1}$ and $\nu(G[Z])=s-p+1$.
   Recall that $c=\lfloor\frac{s-p+1}{p}\rfloor$.
   Hence, we have $cp+d=s-p+1=xp+y(p-1)$.
   Solving for $x$, we get $x=c-y+\frac{d+y}{p}$.
   Since $0\leq d\leq p-1$ and $y\leq p-1$, we have $\frac{d+y}{p}<2$.
   Note that $x\geq 1$, $y\geq 2$, so we conclude that $\frac{d+y}{p}=1$.
    Then $y=p-d$ and $x=c+d+1-p$.
   Hence $$G=(K_{p-1}\vee I_{n-p+1-(c+d+1-p)2p-(p-d)(2p-1)})\cup (c+d+1-p) K_{2p}\cup (p-d)K_{2p-1}=G_5.$$
In this case, in view of $2\leq y=p-d\leq p-1$, we have $1\leq d\leq p-2$.

\begin{mycase}{Case 2.}
  $e(G[X,Z])>0$.
\end{mycase}

  Without loss of generality, assume that there exists a vertex $u\in Z$ such that $uv_1\in E(G)$.
  From \cref{lmm:stability-2} (iii), we have $d(u)\geq p$.
  Since $|X|=p-1$ and $u$ has no neighbours in $Y$, there exists a neighbour of $u$ in $Z$, denoted as $u'$.
  If there exists a vertex $u''\in N(u')\backslash\{u,v_{1},\ldots,v_{p-1}\}$ such that $u'u''\in E(G)$, then according to \cref{lmm:stability-2} (ii), there exist vertices $w_1,\ldots,w_{p-1}\in Y$ such that $P=u''u'uv_1w_1\cdots v_{p-1}w_{p-1}$ is a path of length $2p+1$, contradicting the assumption that $G$ does not contain $P_{2p+1}$.
  Therefore, $u'$ has only one neighbour in $Z$, which is $u$.
  Since $u'\in Z$, we have $d(u')\geq p$, from which we can deduce $N(u')=\{u,v_{1},\ldots,v_{p-1}\}$.
  Similarly, we have $N(u)=\{u',v_{1},\ldots,v_{p-1}\}$.

  Let $Z'=Z\backslash\{u,u'\}$.
  If there is a vertex $z\in Z'$ such that $zv_i\in E(G)$, without loss of generality assuming that $v_i=v_{p-1}$.
  By a similar analysis as before, there must exists a vertex $z'\in Z'$ such that $zz'\in E(G)$.
  Then $P=z'zv_{p-1}w_{p-1}\ldots w_2v_1uu'$ is a copy of $P_{2p+1}$, a contradiction.
  Thus, we have $e(G[X,Z'])=0$.
  Moreover, we have $e(G[\{u,u'\}\cup X\cup Y,Z'])=0$.

  Suppose $C$ is a component of $G[Z']$.
  Similar to the discussion in Case 1, we know that $C$ is a clique with $p+1\leq|C|\leq 2p$.
  Furthermore, we have the following claim.

\begin{claim}
  Every component in $G[Z']$ is a clique of order $2p$.
\end{claim}

\begin{pf}
  Consider $C$ as a component of $G[Z]$ with order at most $2p-1$.
  Now we change the neighbourhood of $u$ to $C$, and denoted by $G'$ the resulting graph.
  We will show that $\nu(G')\le \nu(G)$.
  Recall that $e(G[\{u,u'\}\cup X\cup Y,Z'])=0$.
  We have $\nu(G)=\nu(G[\{u,u'\}\cup X\cup Y])+\nu(G[Z'])=p+\nu(G[Z'])$.
  Clearly, we have $\nu(G')=\nu(G'[\{u'\}\cup X\cup Y])+\nu(G'[\{u\}\cup Z'])\le p-1 +\nu(G[Z'])+1=\nu(G)$.

  Since $|C|\le 2p-1$, $G'$ does not contain a path of length $2p+1$.
  During the transformation from $G$ to $G'$, we removed $\binom{p}{r-1}$ $r$-cliques and added $\binom{|C|}{r-1}$ $r$-cliques.
  Since $|C|\geq p+1$, we have $\binom{|C|}{r-1}>\binom{p}{r-1}$, which contradicts the maximality of $G$.
\end{pf}

  Now we have $G[X\cup Y\cup \{u,u'\}]=K_{p-1}\vee (I_{n-t-p+1}\cup K_2)$ and every component in $G[Z']$ is a clique of order $2p$.
  Let $s-p=c'p+d'$, where $0\leq d'\leq p-1$.
  As previously discussed, it is desirable to have as many independent $2p$-cliques in $G$ as possible.
  Hence, $G[Z']=c' K_{2p}$.
  Then, $$G=(K_{p-1}\vee (K_2\cup I_{n-p-1-2pc'}))\cup c' K_{2p},$$
  where $c'=\lfloor\frac{s-p}{p}\rfloor$.
  Moreover, we claim that $d'< p-1$ in this case.
  Otherwise, we have $d' = p-1$ and $\nu(G) = p + c'p = s - d' = s - (p-1)$.
  We remove $u$, $u'$, and $2p-2$ vertices from $Y$, and add a copy of $K_{2p}$, denoted as $G'$.
  During the transformation from $G$ to $G'$, we removed $2\binom{p-1}{r-1} + \binom{p-1}{r-2} + (2p-2)\binom{p-1}{r-1} = 2p\binom{p-1}{r-1} + \binom{p-1}{r-2}$ $r$-cliques and added $\binom{2p}{r}$ $r$-cliques.
  By \cref{lmm:Vandermonde}, we have
   \begin{align*}
    \binom{2p}{r}
    &=\sum_{i=0}^r\binom{p}{i}\binom{p}{r-i}\\[2mm]
    &\ge\binom{p}{1}\binom{p}{r-1}+\binom{p}{2}\binom{p}{r-2}+\binom{p}{r-1}\binom{p}{1} \\[2mm]
    &>2p\binom{p-1}{r-1}+\binom{p-1}{r-2}.
\end{align*}
  This implies that $G'$ contains more $r$-cliques than $G$.
  Note that $G'$ is $P_{2p+1}$-free and $\nu(G')=(p-1)+c'p+p=s$, which leads to a contradiction.
  Hence $d'< p-1$. Due to $s-(p-1)=c'p+(d'+1)$, we  have
   $c'=\lfloor\frac{s-(p-1)}{p}\rfloor=c$ and $d=d'+1$. Then $$G=(K_{p-1}\vee (K_2\cup I_{n-p-1-2pc}))\cup c K_{2p}=G_6.$$
   In this case,  $d=d'+1\geq 1$. The result follows.
\end{proof}

\bibliographystyle{abbrv}
\bibliography{reference.bib}

\end{document}